\documentclass[10pt,a4paper,reqno]{amsart}
\usepackage[latin1]{inputenc}
\usepackage{amsmath,enumerate,amsthm,graphicx,color,verbatim}
\usepackage{amsfonts,amscd}
\usepackage{latexsym}
\usepackage{amssymb}
\usepackage{breqn}
\newtheorem{thm}{Theorem}

\newtheorem{defn}{Definition}

\newtheorem{lemma}{Lemma}

\newtheorem{prop}{Proposition}

\newcommand{\Z}{\mathbb{Z}}

\newcommand{\R}{\mathbb{R}}

\newcommand{\C}{\mathbb{C}}

\newcommand{\norm}{\vert\vert}
\definecolor{orange}{rgb}{1,0.5,0}
\DeclareMathOperator*{\wlim}{w-lim}
\begin{document}
\author{Milivoje Lukic and Darren C. Ong}
\title[Perturbations of Periodic Schr\"odinger Operators]{Wigner-von Neumann type perturbations of periodic Schr\"odinger Operators}
\begin{abstract}
We consider decaying oscillatory perturbations of periodic \\ Schr\"odinger operators on the half line. More precisely, the perturbations we study satisfy a generalized bounded variation condition at  infinity and an $L^p$ decay condition. We show that the absolutely continuous spectrum is preserved, and give bounds on the Hausdorff dimension of the singular part of the resulting perturbed measure. Under additional assumptions, we instead show that the singular part embedded in the essential spectrum is contained in an explicit countable set. Finally, we demonstrate that this explicit countable set is optimal. That is, for every point in this set there is an open and dense class of periodic Schr\"odinger operators for which an appropriate perturbation will result in the spectrum having an embedded eigenvalue at that point.
\end{abstract}
\thanks{D.O. was supported in part by NSF grant DMS--1067988}
\maketitle

\begin{section}{Introduction}
We consider real solutions $u$ of a Schr\"odinger equation
\begin{equation}\label{perturbation}
Hu:=-u''+(V(x)+V_0(x))u=Eu,
\end{equation}
where $V_0(x)$ is $1$-periodic and $V(x)$ is a decaying and  oscillatory perturbation.

Given a choice of boundary condition $\theta\in [0,\pi)$, our operator $H$ has domain
\begin{align*}
D(H)=&\{u\in L^2(0,\infty) \vert u,u'\in AC_{loc}, -u''+(V+V_0)u\in L^2,\\
& u'(0)\sin(\theta)=u(0)\cos(\theta)\}.
\end{align*}
The operator $H$ is self-adjoint, has $0$ as a regular point and is limit point at $+\infty$.

For $z\in \C$ with $\mathrm{Im} z>0$, there is a nontrivial solution of -$u''_z+(V+V_0)u_z=zu_z$ which is $L^2$ near $+\infty$. We then define the $m$-function corresponding to $H$
\[m(z)=\frac{u_z'(0)\cos(\theta)+u_z(0)\sin(\theta)}{u_z(0)\cos(\theta)-u_z'(0)\sin(\theta)},\]
and the canonical spectral measure
\[d\mu=\frac{1}{\pi}\wlim_{\epsilon\downarrow 0}m(x+i\epsilon)dx.\]

  We give $V$ as

\begin{equation}\label{Vdef}
V(x)=\sum_{l=1}^\infty c_le^{-i\phi_l x}\gamma_l(x),
\end{equation}
 where the $c_l, \phi_l,\gamma_l$ are chosen so $V(x)$ is real, and so that it obeys the conditions of Theorem 1.1 of \cite{Lukic-infinite}. That is,
 
 \begin{enumerate}[(i)]
\item the $\gamma_l(x)$ are functions of uniformly bounded variation in $l$: in other words,
\begin{equation} \label{boundedvar}
\sup_l \mathrm{Var} (\gamma_l, (0,\infty))<\infty.
\end{equation}
\item for some $p\in \Z$, $p\geq 2$,
\begin{equation}\label{Lp}
\sup_l\int_0^\infty  \vert \gamma_l(x)\vert  ^p dx<\infty.
\end{equation}
\item for some $\mathfrak a\in \left(0, \frac{1}{p-1}\right)$,
\begin{equation}\label{alphacondition}
\sum_{l=1}^\infty \vert c_l\vert ^{\mathfrak a}<\infty.
\end{equation}
\end{enumerate}
We refer to these decaying oscillatory perturbations as Wigner-von Neumann type perturbations based on the Wigner von Neumann potential function on $(0,\infty)$:
 \[ V(x)=-8\frac{\sin(2x)}{x}+O(x^{-2}), x\to\infty.\]
 Historically, this potential was interesting because the corresponding Schr\"odinger operator had the unexpected property that there is an eigenvalue at $+1$ embedded in the absolutely continuous spectrum. 

Let us first discuss solutions of the unperturbed Schr\"odinger equation
\begin{equation}\label{unperturbed}
H_0\varphi:=-\varphi''+V_0(x)\varphi=E\varphi,
\end{equation}
where $\varphi$ is the Floquet solution. We express $\varphi$ as $p(x)e^{ikx}$, where $k$ is the quasimomentum. We also write $\varpi(x)=\mathrm{Arg}$  $ p(x)$. When we wish to emphasize the dependence of $k$ on $E$, we will write $k_E$. It is known that the essential spectrum of a periodic Schr\"odinger operator consists of a union of absolutely continuous closed intervals (often referred to as ``bands"). Any two of those bands can intersect at most at a point. Additionally, by Weyl's theorem, $\sigma_{ess}(H)=\sigma_{ess}(H_0)$.

Let us define the set $S$ so that
\[ S=\{ E\in\mathrm{Int}(\sigma_{ess}(H))\vert \text{not all solutions of (\ref{perturbation}) are bounded}\}.\]
Due to standard results in spectral theory (\cite{Gilbert-Pearson},\cite{Behncke},\cite{Stolz}), we know that $\mu$ is mutually absolutely continuous with Lebesgue measure on $\mathrm{Int}(\sigma_{ess}(H))\setminus S$.

Our theorem for the finite frequency case (i.e. all but finitely many of the $c_l$s are zero) is given as
\begin{thm}\label{finitefreqthm}
Considering  $V$ chosen so that (\ref{Vdef}) is a finite sum, $S$ is contained in
\begin{equation}\label{finitefreqset}\left\{ E\in \mathrm{Int}(\sigma_{ess}(H)) \Bigg| 
\pm 2k_E\in \bigcup_{l=1}^{p-1} A\oplus \ldots\oplus A\mod 2\pi \text{(sum of $l$ $A$'s mod $2\pi$)}\right\},
\end{equation}
where $A$ is the set of all $\phi_l$s. This implies that each band of $\sigma_{ess}(H)$ has at most finitely many embedded pure points, that $H$ has no singular continuous spectrum, and that the
absolutely continuous spectrum of $H$ is equal to $\sigma_{ess}(H).$
\end{thm}
Conversely, we can also say
\begin{thm}\label{embeddedthm}
Fix $p$ and the size of $A=\{\phi_l\}$. For any choice of $A$ away from an
algebraic set of codimension $1$, if we fix an energy $E$ and a
quasimomentum $k_E$ in (\ref{finitefreqset}), among all the $V_0 \in
L^1_{loc}((0,1))$ for which $k_E$ corresponds to $E$, there is an open
and dense set of $V_0$ such that there is a choice of $V$ for which
the spectrum of $H$ has an embedded pure point at $E$.
\end{thm}
Our theorem for the infinite frequency case is as follows.
\begin{thm}\label{infinitefreqthm}
Let the potential $V$ be given by (\ref{Vdef}).
Assume that for some $\mathfrak a\in (0,1/(p-1))$, $\sum_{k=1}^\infty \vert c_k\vert^{\mathfrak a}<\infty$. Then the Hausdorff dimension of $S$ is at most $(p-1)\mathfrak a$. 
Moreover, on $\mathrm{Int}(\sigma_{ess}(H)) \setminus S$, the spectral
measure $\mu$ is mutually absolutely continuous with Lebesgue measure,
so the absolutely continuous spectrum of $H$ is equal to
$\sigma_{ess}(H)$.
\end{thm}
 These results are an extension of \cite{Lukic-continuous} and \cite{Lukic-infinite}, which study the case when our $V_0$ is identically zero. We may also see these results as extensions of \cite{Kurasov-Naboko}, which addresses the problem in $L^2$ for a less general $V$, using different methods.
 
The next three sections will explain the proofs of Theorems \ref{finitefreqthm} and \ref{infinitefreqthm}. In the last section we will discuss a converse problem, that is, the existence of embedded eigenvalues in the case where $(\ref{Vdef})$ is a finite sum. This section will culminate in a proof of Theorem \ref{embeddedthm}.
\end{section}
\begin{section}{Preliminary Lemmas}
We wish to control solutions of the perturbed Schr\"odinger operator
(\ref{perturbation}) by comparing them to solutions of the unperturbed operator (\ref{unperturbed}), so
we will use modified Pr\"ufer variables which were defined in \cite{KRS}
for this purpose.

We first fix $E$ and fix a $\varphi$. We will need to choose $\varphi$ so that it is linearly independent of $\overline\varphi$: this is possible since we restrict our attention to $E$ in the interior of $\sigma_{ess}$. We define $\rho(x)\in \C$ as in (21) of \cite{KRS}:

\[
\left(
\begin{array}{c}
u'(x)\\
u(x)
\end{array}
\right)=
\mathrm{Im} \left[
\rho(x) 
\left(
\begin{array}{c}
\varphi'(x)\\
\varphi(x)
\end{array}
\right)
\right].
\]
 We also define

\begin{align}
R(x)=&\vert\rho(x)\vert\\
\eta(x)=&\mathrm{Arg}(\rho(x))\\
\theta(x)=&kx+\varpi(x)+\eta(x).
\end{align}

We choose $\eta$ so that $\eta(0)\in (-\pi,\pi]$ and $\eta$ continuous. Our Pr\"ufer variables are going to be $R$ and $\eta$.

Let $\omega$ be the Wronskian of $\varphi, \overline{\varphi}$. Note that we can assume the Wronskian is positive and independent of $x$. By Theorem 2.3 of \cite{KRS} we have that
\[
[\ln R(x)]'+i\eta'(x)=\frac{\rho'(x)}{\rho(x)}=\frac{2\vert \varphi(x)\vert^2}{\omega} V(x) \sin (kx+\varpi(x)+\eta(x))e^{-ikx-i\varpi(x)-i\eta(x)},
\]
which we then rewrite as
\[-\eta'(x)+i[\ln(R(x)]'=\frac{\vert\varphi(x)\vert^2}{\omega}V(x)(1-e^{-2i(kx+\varpi(x)+\eta(x))}).
\]
In particular, we can write
\begin{align}
\label{R}[\ln R(x)]'=&\mathrm{Im}\left(\frac{\vert \varphi(x)\vert^2}{\omega}V(x)e^{2ikx+2i\varpi(x)+2i\eta(x)}\right),\\
\label{eta}\eta'(x)=&\frac{\vert \varphi(x)\vert^2}{\omega}V(x)\left( -1+\frac{1}{2}\left(e^{2ikx+2i\varpi(x)+2i\eta(x)}+e^{-2ikx-2i\varpi(x)-2i\eta(x)}\right)\right).
\end{align}

Note that our choice of $V(x)$ given in (\ref{Vdef}) means that $[\ln(R(x)]'$ can be written as a sum of terms which are products of periodic and decaying factors. Our immediate goal is to perform a sequence of manipulations that look like integration by parts so as to integrate the periodic factors of $[\ln(R(x)]'$ and differentiate the decaying factors, an approach in the spirit of \cite{Lukic-continuous} and \cite{Lukic-infinite}. First, we need a technical lemma to choose the appropriate antiderivative for the periodic factors.

\begin{prop}\label{lambda}
Let $\Phi(\alpha;x)$ be continuous and $1$-periodic in $x$. Then there exists a continuous $1$- periodic function $\tilde\Phi_\alpha(x)$ such that \[(i\tilde\Phi_\alpha(x) e^{i\alpha x})'=(1-e^{i\alpha})\Phi(\alpha;x) e^{i\alpha x}.\]
Furthermore, if $\alpha\not\equiv 0\mod 2\pi$, this function is unique.
\end{prop}
\begin{proof}
Let $Q_\alpha(x)$ be the antiderivative of $\Phi(\alpha;x) e^{i\alpha x}$ such that $Q_\alpha(0)=0$. We then define

\begin{equation}\label{tildePhi}
\tilde\Phi_\alpha(x)= -i Q_\alpha(x)e^{-i\alpha x}(1-e^{i\alpha})+i Q_\alpha(1)e^{-i\alpha x}.
\end{equation}
We first note that $(i\tilde\Phi_\alpha(x) e^{i\alpha x})'=(1-e^{i\alpha})\Phi(\alpha;x) e^{i\alpha x}$. It remains to show that $\tilde\Phi_\alpha(x)$ is $1$-periodic. 

We observe that since $i\tilde\Phi_\alpha(x) e^{i\alpha x}$ is an antiderivative of $(1-e^{i\alpha})\Phi(\alpha;x) e^{i\alpha x}$, and since 

\[
\tilde\Phi_\alpha(1)=-i Q_\alpha(1)e^{-i\alpha}(1-e^{i\alpha}) +i Q_\alpha(1) e^{-i\alpha}\\
= i Q_\alpha(1)\\
= \tilde \Phi_\alpha (0),
\]
we must then have
\begin{align*}
i\tilde\Phi_\alpha(x)e^{i\alpha x}-i\tilde\Phi_\alpha(0)=&(1-e^{i\alpha})\int_0^x \Phi(\alpha;t) e^{i\alpha t} dt\\
=& (e^{-i\alpha}-1)\int_0^x \Phi(\alpha;t+1) e^{i\alpha (t+1)} dt \text{  ($1$-periodicity of $\Phi$)}\\
=& (e^{-i\alpha}-1)\int_1^{x+1} \Phi(\alpha;t) e^{i\alpha t} dt\\
=& e^{-i\alpha}\left(i\tilde\Phi_\alpha(x+1)e^{i\alpha(x+1)} - i\tilde\Phi_\alpha(1)e^{i\alpha} \right)\\
i\tilde\Phi_\alpha(x)e^{i\alpha x}=& i\tilde\Phi_\alpha(x+1)e^{i\alpha x} - i\tilde\Phi_\alpha(1) +i\tilde\Phi_\alpha(0)\\
\tilde\Phi_\alpha(x)=& \tilde\Phi_\alpha(x+1). 
\end{align*}
We now demonstrate uniqueness when $\alpha\not\equiv 0\mod 2\pi$. Consider a $1$-periodic continuous function $\Psi(x)$ such that $(\Psi(x) e^{i\alpha x})'=(1-e^{i\alpha})\Phi(\alpha;x) e^{i\alpha x}$. We then know that for some constant $C$,
\[ \Psi(x) e^{i\alpha x}=\tilde\Phi_\alpha (x) e^{i\alpha x}+C\]
we can then write
\[\Psi(x)=\tilde\Phi_\alpha (x) +Ce^{-i\alpha x}.\]

Observe that $e^{-i\alpha x}$ is not $1$-periodic since $\alpha$ cannot be a multiple of $2\pi$, by hypothesis. Thus it must be true that $C=0$, and so $\Psi(x)=\tilde \Phi_\alpha(x)$.
\end{proof}

\begin{paragraph}{\textbf{Remark}}
It is easy to see that if $\alpha\equiv 0\mod 2\pi$, we may define $\tilde\Phi_\alpha (x)$ to be an arbitrary constant. The following lemma demonstrates that there is a natural choice for this arbitrary constant.
\end{paragraph}
\begin{lemma}\label{continuityat0}
Let $n$ be an integer. Assume that a function $\Phi(\alpha;x)$ continuous and $1$-periodic in $x$, and converges uniformly as $\alpha\to 2\pi n$ for $x$ on $\R$ . Then $Q_{2\pi n}(1)=\lim_{\alpha\to 2\pi n} Q_\alpha(1)$ exists, and $\tilde\Phi_\alpha(x)$ converges to $iQ_{2\pi n}(1)$ as $\alpha\to 2\pi n$ uniformly for $x$ on $\R$
\end{lemma}
\begin{proof}
Let us define, as in Proposition \ref{lambda}, $Q_\alpha(x)$ as the antiderivative of $\Phi(\alpha;x)e^{i\alpha x}dx$ with $Q_\alpha(0)=0$. Then, due to the fact that $\Phi(\alpha;x), e^{i\alpha x}$ uniformly converge as $\alpha\to 2\pi n$ for $x$ in compact subsets of $\R$,

\begin{align*}
Q_\alpha (x)=&\int_0^x\Phi(\alpha;t)e^{i\alpha t} dt,\\
\lim_{\alpha\to 2\pi n }\int_0^x\Phi(\alpha;t)e^{i\alpha t} dt=&\int_0^x\lim_{\alpha\to {2\pi n}}\Phi(\alpha;t)e^{i\alpha t}dt\\
=&\int_0^x \Phi(2\pi n;t)dt\\
=& Q_{2\pi n}(x).
\end{align*}
Thus indeed $\lim_{\alpha\to 2\pi n}Q_\alpha(x)$ converges to $Q_{2\pi n}(x)$. Furthermore, since \[\int_{x_0}^{x_1}\Phi(\alpha;t)e^{i\alpha t} dt,\] is bounded for $x_0<x_1$ and $x_1$ bounded, the convergence is uniform on compact subsets of $x$. From (\ref{tildePhi}), we then find that $\tilde \Phi_\alpha(x)$ converges uniformly on compact subsets of $\R$. Since $\tilde\Phi_\alpha(x)$ is $1$-periodic, it follows that it in fact converges uniformly on $\R$.

\end{proof}

As a consequence, defining $\tilde \Phi_{2\pi n}(x)$ as $iQ_{2\pi n}(1)$ ensures that $\tilde\Phi_\alpha(x)$ is uniformly continuous at $\alpha\equiv 0\mod 2\pi$.

Let us define $C_{per}(0,\infty)$ as the space of continuous $1$-periodic functions on $(0,\infty)$ with the uniform norm. Then the preceding proposition allows us to make the following definition:
\begin{defn}\label{lambdadefn}  For $K\in \Z, \alpha\in \R$,  $\lambda_{\alpha,K}$ is a linear operator from $C_{per}(0,\infty)$ to itself so that $\lambda_{\alpha,K}$ takes each $\Phi(\alpha;x)\in C_{per}(0,\infty)$ to the corresponding \[\tilde \Phi_\alpha(x) e^{-2Ki\varpi(x)},\] where $\tilde \Phi_\alpha$ is as defined in (\ref{tildePhi}). 
\end{defn}
\begin{lemma}\label{lambdanorm} With $\norm \ldots\norm$ denoting the operator norm, $\norm \lambda_{\alpha,K}\norm \leq 2.$
\end{lemma}
\begin{proof}
Let $\Phi(\alpha;x)\in C_{per}(0,\infty)$.
Note that since $\vert(\lambda_{\alpha,K} \Phi)(x)\vert$ is $1$-periodic and continuous, it must have a maximum in $[0,1)$. Let $Q_\alpha(x)$ once again be the antiderivative of $\Phi(\alpha;x)e^{i\alpha(x)}$ such that $Q_\alpha(0)=0$. It is clear that $\vert Q_\alpha(x)\vert \leq \norm \Phi\norm$ for any $x\in [0,1)$. In fact, it is also clear that $\vert Q_\alpha(1)-Q_\alpha(x)\vert \leq \norm \Phi\norm$. When we rewrite (\ref{tildePhi}) as
\[\tilde \Phi_\alpha (x) =i e^{i\alpha}Q_\alpha(x)e^{-i\alpha x}+ i (Q_\alpha(1)-Q_\alpha(x))e^{-i\alpha x},
\]
it follows that $\norm\tilde \Phi_\alpha\norm\leq 2$.
\end{proof}

Since the variations of the $\gamma_i$ are uniformly bounded, it is possible to define
\[
\tau=\sup_l \mathrm{Var}(\gamma_l,(0,\infty))<\infty.
\]

\begin{lemma}\label{fubini}
We let $J,K\in \Z$ with $J\geq 1$ and $0\leq K\leq J$, and $0\leq a<b<\infty$. We also define $\Gamma(x)=\gamma_{m_1}(x)\ldots\gamma_{m_J}(x)$, and we let $\phi$ be some phase in $[0,2\pi)$. Then where $\Phi$ is some continuous $1$-periodic function,
\begin{dmath}
2 \norm\lambda_{2Kk-\phi,K} \Phi\norm \tau^J\geq \left \vert \int_a^b\left( (1-e^{i(2Kk-\phi)})e^{2Ki(\eta(x)+kx)}e^{-i\phi x}\Gamma(x)\Phi(x)\\
-2Ke^{2Ki(\eta(x)+kx+\varpi(x))}e^{-i\phi x}\Gamma(x) \dfrac{d\eta(x)}{dx} \lambda_{2Kk-\phi,K}\Phi(x)\right)\right dx\vert.
\end{dmath}
\end{lemma}
\begin{proof}
The proof is identical to the proof of Lemma 2.1 of \cite{Lukic-infinite}, except that we redefine $\psi(x)$ in that proof to be equal to \[e^{2Ki\eta(x)}\cdot ie^{i(2Kkx+2K\varpi(x)-\phi x)}\lambda_{2Kk-\phi,K} \Phi(x).\]
 Keep in mind that the derivative of $ie^{i(2Kkx+2K\varpi(x)-\phi x)} \lambda_{2Kk-\phi,K} (\Phi(x))$ is \newline
 $(1-e^{i(2Kk-\phi)}) e^{(2Kk-\phi)ix}\Phi(x)$, by the definition of  $\lambda_{2Kk-\phi,K}$.
\end{proof}

\end{section}
\begin{section}{A recursion relation}
 For integers $J,K$ with $J\geq 1 $ and $0\leq K\leq J$, let us now define functions $f_{J,K}, g_{J,K}$ as follows. They are functions of $1+J$ variables, $x ,\phi_1,\ldots, \phi_J$, and they also depend implicitly on $E$. We will use them in (\ref{S}), (\ref{S11}) to estimate the growth of $R(x)$.
 
For convenience, we define $\Phi_0(x)=\vert \varphi(x)\vert^2/\omega$.
 We first set
\begin{equation} \label{fginitial}
 f_{1,0}(x;\phi_1)=0, f_{1,1}(x;\phi_1)=\Phi_0(x).
  \end{equation}
  We then define 
  \begin{equation}\label{g}
  g_{J,K}(x;\{\phi_j\}_{j=1}^J)=\frac{2K}{1-e^{i(2Kk-\sum_{j=1}^J \phi_j)}}\lambda_{2Kk-\sum_{j=1}^J \phi_j,K}[e^{2iK\varpi(x)}f_{J,K}(x;\{\phi_j\}_{j=1}^J)],
  \end{equation}
  and for $J\geq 2$,
  \begin{equation}
  f_{J,K}(x;\{\phi_j\}_{j=1}^J)=\sum_{l=K-1}^{K+1}\sum_{\sigma\in S_J}\frac{ \Phi_0(x)}{J!} w_{K-l}g_{J-1,l}(x;\{\phi_{\sigma (j)}\}_{j=1}^{J-1}).\label{f}
    \end{equation}
    Here $S_J$ denotes the symmetric group in $J$ elements and, motivated by (\ref{eta}), we define the constant function
    \begin{equation}
    w_a(x;\phi)=\begin{cases}
-1& a=0,\\
\frac{1}{2}& a=\pm 1,\\
0 &\vert a\vert\geq 2.
\end{cases}
\end{equation}
We use the symmetric product defined as Definition 2.1 of \cite{Lukic-infinite}, with some slight modifications:

\begin{defn}
For a function $p_I$ of $I+1$ variables and a function $q_J$ of $J+1$ variables, their symmetric product is a function $p_I \odot q_J$ of $1+I+J$ variables defined by
 
 \[ (p_I \odot q_J)(x;\{ \phi_i\}_{i=1}^{I+J} )=\frac{1}{(I+J)!} \sum_{\sigma\in S_{I+J}}p_I(x;\{ \phi_{\sigma(i)}\}_{i=1}^I)q_J(x;\{ \phi_{\sigma(i)}\}_{i=I+1}^{I+J}).\]
\end{defn}

Where $\delta$ refers to the Kronecker delta, we can express $f_{J,K}$ as 

\begin{equation}\label{fodot}
f_{J,K}=\delta_{J-1} \Phi_0(x)+\sum_{a=-1}^1 \Phi_0(x)w_a\odot g_{J-1,K+a}.
\end{equation}

\begin{lemma}
For $0\leq K\leq I$ and $0<l<I$,
\begin{align}\label{frecurse}
f_{I,K}=\frac{1}{2}\sum_{j=0}^I f_{j,l} \odot g_{I-j,K-l}.\\
\label{grecurse}g_{I,K}=\frac{1}{2}\sum_{j=0}^I g_{j,l} \odot g_{I-j,K-l}.
\end{align}
\end{lemma}
\begin{proof}
Let us assume for now that the $2Kk-\sum \phi$ terms are never congruent to $0\mod 2\pi$.
We prove both (\ref{frecurse}) and (\ref{grecurse}) at the same time, inductively. The statement is vacuously true for $I\leq 1$. Let us assume now that it holds for $I-1$. We use (\ref{fodot}) and notice
\[
\sum_{t=0}^I f_{t,l}\odot g_{I-t,K-l}=\sum_{t=0}^I(\delta_{t-1} \Phi_0(x)+\sum_{a=-1}^1 \Phi_0(x)w_a\odot g_{t-1,l+a})\odot g_{I-t,K-l}.
\]

Using the inductive assumption, we may apply (\ref{grecurse}) to the $g\odot g$ terms, unless $l+a\leq 0$. But $l+a\leq 0$ holds only for $l=1, a=-1$, and in this exceptional case $g_{t-1, l+a}=0$.  Thus,

\begin{align*}
\sum_{t=0}^I f_{t,l}\odot g_{I-t,K-l}=&\sum_{t=0}^I \delta_{t-1}\Phi_0(x)\odot g_{I-t,K-l}\\
&+\sum_{a=-1}^1\sum_{t=0}^I \Phi_0(x)w_a\odot g_{t-1,l+a}\odot g_{I-t,K-l}\\
=&\delta_{l-1}\Phi_0(x)\odot g_{I-1,K-1}\\
&+\sum_{a=-1}^1\sum_{t=0}^I \Phi_0(x)w_a\odot 2
 (g_{I-1,K+a}-\delta_{a+1}\delta_{l-1}  g_{I-1, K-1})\\
 &\text{(because $t=1$, $l>1$ implies $f_{t,l}=0$)}\\
 =&\delta_{l-1}\Phi_0(x)\odot g_{I-1,K-1}\\
&+2f_{I,K}-\delta_{I-1}\Phi_0(x)-2w_{-1}\Phi_0(x)\odot\delta_{l-1}  g_{I-1, K-1}\\
=&2 f_{I,K}.
\end{align*}
The last equality is due to the fact that we are assuming $I>1$.
It remains to prove (\ref{grecurse}).

We can calculate that for any $\sigma\in S_I$, using the product rule and (\ref{g})
\begin{dmath*}
\dfrac{d}{dx}\left(\frac{ie^{i(2K(kx+\varpi(x))-\sum_{j=1}^I \phi_jx)}}{2}\sum_{t=0}^I g_{t,l}(\{x; \phi_{\sigma(j)}  \}_{j=1}^t)  g_{I-t,K-l}(\{x; \phi_{\sigma(j)}  \}_{j=t+1}^I)\right)\\
=\dfrac{d}{dx}\left(\frac{i}{2}\sum_{t=0}^I \left[ e^{i(2l(kx+\varpi(x))-\sum_{j=1}^t \phi_{\sigma(j)}x)} g_{t,l}(\{ x;\phi_{\sigma(j)}  \}_{j=1}^t)\right]\\
\times\left[ e^{i(2(K-l)(kx+\varpi(x))-\sum_{j=t+1}^I \phi_{\sigma(j)}x)} g_{I-t,K-l}(x;\{ \phi_{\sigma(j)}  \}_{j=t+1}^I)\right]\right).\\
\end{dmath*}
Recalling (\ref{g}), Proposition \ref{lambda} and Definition \ref{lambdadefn} and applying the product rule, we find that this expression is equal to
\begin{dmath*}
\left(\frac{1}{2}\sum_{t=0}^I \left[ 2l e^{i(2l(kx+\varpi(x))-\sum_{j=1}^t \phi_{\sigma(j)}x)} f_{t,l}(\{x; \phi_{\sigma(j)}  \}_{j=1}^t)\right]\\
\times\left[ e^{i(2(K-l)(kx+\varpi(x))-\sum_{j=t+1}^I \phi_{\sigma(j)}x)} g_{I-t,K-l}(x;\{ \phi_{\sigma(j)}  \}_{j=t+1}^I)\right]\right)\\
+\left(\frac{1}{2}\sum_{t=0}^I  \left[ e^{i(2l(kx+\varpi(x))-\sum_{j=1}^t \phi_{\sigma(j)}x)} g_{t,l}(x;\{ \phi_{\sigma(j)}  \}_{j=1}^t)\right]\\
\times\left[ 2(K-l) e^{i(2(K-l)(kx+\varpi(x))-\sum_{j=t+1}^I \phi_{\sigma(j)}x)} f_{I-t,K-l}(x;\{ \phi_{\sigma(j)}  \}_{j=t+1}^I) \right]\right).
\end{dmath*}
Collecting terms, we find that this in turn can be written as
\begin{dmath*}
\frac{e^{i(2K(kx+\varpi(x))-\sum_{j=1}^I \phi_jx)}}{2}\left(\sum_{t=0}^I  2l f_{t,l}(x;\{ \phi_{\sigma(j)}  \}_{j=1}^t) g_{I-t,K-l}(x;\{ \phi_{\sigma(j)}  \}_{j=t+1}^I)+2(K-l) g_{t,l}(x;\{ \phi_{\sigma(j)}  \}_{j=1}^t)   f_{I-t,K-l}(x;\{ \phi_{\sigma(j)}  \}_{j=t+1}^I) \right).\\
\end{dmath*}
Once we average across all permutations $\sigma\in S_I$, we arrive at

\begin{align*}
&\left(\sum_{t=1}^I\frac{ie^{i(2K(kx+\varpi(x))-\sum_{t=1}^I \phi_jx)}}{2} g_{t,l}\odot g_{I-t,K-l}\right)'\\
=&2l\frac{e^{i(2K(kx+\varpi(x))-\sum_{j=1}^I \phi_jx)}}{2}\sum_{t=1}^I f_{t,l}\odot g_{I-t,K-l}\\
&+2(K-l)\frac{e^{i(2K(kx+\varpi(x))-\sum_{j=1}^I \phi_jx)}}{2}\sum_{t=1}^I f_{t,K-l}\odot g_{I-t,l}\\
=&2Ke^{i(2K(kx+\varpi(x))-\sum_{j=1}^I \phi_jx)} f_{I,K}.
\end{align*}
Since $\frac{1}{2}g_{t,l}\odot g_{I-t,K-l}$ is $1$-periodic in $x$, we conclude by uniqueness in Proposition \ref{lambda} that it is equal to $g_{I,K}$.

Now in the case where some $2Kk-\sum\phi\equiv0\mod 2\pi$, we may apply Lemma \ref{continuityat0} to assert that $f,g$ are continuous at $\alpha\equiv 0\mod 2\pi $ and thus since the equalities (\ref{frecurse}),(\ref{grecurse}) hold for every $2Kk-\sum\phi$ in a neighborhood of $0\mod 2\pi$, they must be true for $2Kk-\sum\phi\equiv 0\mod 2\pi$ as well.
\end{proof}

Let us define functions $h_j$ of $1+j$ variables recursively by $h_0(x)=1$ and 

\begin{align*}
&h_J(x;\phi_1,\ldots \phi_J)\\
=&\frac{1}{\vert 1-e^{i(2k-\phi_1-\ldots-\phi_J)}\vert}\sum_{j=0}^{J-1} h_j (x;\phi_1,\ldots, \phi_j)h_{J-j-1}(x; \phi_{j+1},\ldots, \phi_{J-1}).
\end{align*}

Next, we recall that by Lemma \ref{lambdanorm}, $\norm \lambda_\alpha\norm \leq 2$. 

\begin{lemma}\label{h-glemma}
Where $\norm \Phi_0\norm$ refers to the maximum of the periodic continuous function $\vert \Phi_0(x)\vert $, the function $g_{J,1}$ can be bounded in terms of $h_J$ in the following manner:
\[\vert g_{J,1}(x;\{\phi_j\}_{j=1}^J)\vert\leq \frac{2(2 \norm \Phi_0\norm)^J}{J!}\sum_{\sigma\in S_J} h_J(x;\{\phi_{\sigma(j)}\}_{j=1}^J ).\]
\end{lemma}
\begin{proof}

We prove this by induction. First we note that 
\[ \vert g_{1,1}(x;\phi_1)\vert =\left\vert 2\frac{\lambda_{2k-\phi_1}e^{2i\varpi(x)}f_{1,1}(x;\phi_1)}{1-e^{i(2k-\phi_1)}}\right\vert\leq 4\norm \Phi_0\norm h_1(x;\phi_1).\]

For $J\geq 2$, we deduce, using (\ref{f}), (\ref{fodot}) and the inductive hypothesis,

\begin{dmath*}
\vert g_{J,1}(x;\{\phi_j\}_{j=1}^J )\vert =\left\vert 2\frac{\lambda_{2k-\sum_{j=1}^J \phi_j}e^{2i\varpi(x)}f_{J,1}}{1-e^{i(2k-\sum_{j=1}^J \phi_j)}}\right\vert\\
= \left\vert\frac{2}{1-e^{i(2k-\sum_{j=1}^J \phi_j)}} \\
\times\lambda_{2k-\sum_{j=1}^J \phi_j}\left(e^{2i\varpi(x)}\Phi_0(x) \left[ w_0\odot g_{J-1,1}+\frac{1}{2}w_1\odot \sum_{j=1}^{J-2} g_{j,1}\odot g_{J-j-1,1} \right]\right)\right\vert\\
\leq  \frac{4\norm \Phi_0\norm}{\vert 1-e^{i(2k-\sum_{j=1}^J \phi_j)}\vert}\left\vert\left(  w_0\odot g_{J-1,1}+\frac{1}{2}w_1\odot \sum_{j=1}^{J-2} g_{j,1}\odot g_{J-j-1,1} \right)\right\vert\\
\leq  \frac{2(2\norm \Phi_0\norm )^J}{\vert1-e^{i(2k-\sum_{j=1}^J \phi_j)}\vert}\left\vert\left(  2w_0\odot h_{J-1,1}+2w_1\odot \sum_{j=1}^{J-2} h_{j,1}\odot h_{J-j-1,1} \right)\right\vert\\
\leq  \frac{2(2\norm \Phi_0\norm )^J}{\vert 1-e^{i(2k-\sum_{j=1}^J \phi_j)}\vert}\left(  2\vert w_0\vert \odot h_{J-1,1}+2w_1\odot \sum_{j=1}^{J-2} h_{j,1}\odot h_{J-j-1,1} \right)\\
=\frac{2(2 \norm \Phi_0\norm)^J}{J!}\sum_{\sigma\in S_J} h_J(x;\{\phi_{\sigma(j)}\}_{j=1}^J ).
\end{dmath*}
\end{proof}

We are now prepared to prove the following lemma:

\begin{lemma}\label{mainlemma}
 Let $E\in (0,\infty)$ so that $V$ can be given by (\ref{Vdef}), and that the following conditions are satisfied:
 \begin{enumerate}[(i)]
 \item $\sum_{l=1}^\infty \vert c_l\vert <\infty$
 \item for $j=1,\ldots p-1$, and $1\leq K\leq j$,
 \begin{equation}\label{smalldivisor}
 \sum_{l_1,\ldots, l_j=1}^\infty \vert c_{l_1}\ldots c_{l_j} h_j(x;\phi_{l_1},\ldots, \phi_{l_j})\vert<\infty.
 \end{equation}

 \end{enumerate}
 then all the solutions of (\ref{perturbation}) are bounded.
\end{lemma}

 Before we prove this lemma, we first define 
 \begin{equation}\label{S}
 \mathcal S_{J,K}(x)=\sum_{m_1,\ldots, m_J}^\infty f_{J,K}(\phi_{m_1},\ldots, \phi_{m_J}))\beta_{m_1}(x)\ldots\beta_{m_J}(x)e^{2iK[kx+\varpi(x)+\eta(x)]},
 \end{equation}
  where $\beta_l(x)=c_le^{-i\phi_l x}\gamma_l(x)$.
  
  We can then rewrite (\ref{R}) as 
  \begin{equation}\label{S11}
  \ln R(b)-\ln R(a)=\mathrm{Im} \int_a^b\mathcal S_{1,1}(x) dx. 
  \end{equation}
  We also define 
  \begin{equation}
  E_{J,K} =\sum_{m_1,\ldots, m_J=1}^\infty \vert c_{m_1}\ldots c_{m_J} \vert\cdot \norm g_{J,K} (\phi_{m_1},\ldots, \phi_{m_J})\norm,
  \end{equation}
  where $\norm g\norm $ refers to the maximum of the continuous periodic function $\vert g\vert $.

  \begin{lemma}
Assume the hypotheses of Lemma \ref{mainlemma}. For $J=1,\ldots, p-1$,
  \begin{equation}\label{induction}
  \left\vert\int_a^b \left( \sum_{K=1}^J \mathcal S_{J,K}-\sum_{K=0}^{J+1} \mathcal S_{J+1,K}\right) dx\right\vert \leq \sum_{K=1}^J \frac{E_{J,K}^J\tau^J}{K}
  \end{equation}
  
   \end{lemma}
   \begin{proof}
We apply Lemma \ref{fubini}, setting $\Phi=e^{2iK\varpi(x)}f_{J,K}$, with $\Gamma(x)=\gamma_{m_1}\ldots\gamma_{m_J}$ and $\phi=\phi_{m_1}+\ldots+\phi_{m_J}$. Noting that the assumption (\ref{smalldivisor}) implies $2Kk-\phi\not\equiv 0\mod 2\pi$, we obtain

\begin{dmath}
 \frac{\norm  g_{J,K}\norm}{K}\tau^J\geq \left \vert \int_a^b\left( e^{2Ki(\eta(x)+kx+\varpi(x))}e^{-i\phi x}\Gamma(x)f_{J,K}\\
-e^{2Ki(\eta(x)+kx+\varpi(x))}e^{-i\phi x}\Gamma(x) \dfrac{d\eta(x)}{dx}g_{J,K}\right)\right dx\vert.
\end{dmath}
We then expand $d\eta/dx$ using (\ref{eta}), apply (\ref{f}), multiply by $c_{m_1}\ldots c_{m_J}$, sum in $m_1,\ldots, m_J $ from $1$ to $\infty$, and sum in $K$ from $1$ to $J$ to prove the lemma.
   \end{proof}
   
 Let us define
\begin{equation}
\mathfrak m=\sup_l\norm \gamma_l(x)\norm_p.
\end{equation}
We know it is finite by assumptions we placed on $V$.
  
\begin{lemma}\label{convergence}
Assume the hypotheses of Lemma \ref{mainlemma}. $\mathcal S_{J,K}(x)$ is absolutely convergent when $1\leq K\leq J\leq p$, and if in addition $J\geq 2$ then
\end{lemma}  
\begin{align}
&\sum_{m_1,\ldots, m_J=1}^\infty \vert f_{J,K} (\phi_{m_1},\ldots, \phi_{m_J})\beta_{m_1}(x)\ldots \beta_{m_J}(x)\vert \nonumber\\
& \leq \vert\Phi_0(x)\vert \sum_{a=-1}^1 \vert w_a\vert E_{J-1, K+a}\sum_{l=1}^\infty \vert c_l\vert  \tau^J. \label{convergenceJ<p}
\end{align}
Furthermore, if $J=p$, we have that also

\begin{align}
&\int_0^\infty\sum_{m_1,\ldots, m_J=1}^\infty \vert f_{J,K} (\phi_{m_1},\ldots, \phi_{m_J})\beta_{m_1}(x)\ldots \beta_{m_J}(x)\vert dx\nonumber\\
& \leq \vert\Phi_0(x)\vert \sum_{a=-1}^1 \vert w_a\vert E_{p-1, K+a}\sum_{l=1}^\infty \vert c_l\vert \mathfrak m(x)^J.\label{convergenceJ=p}
\end{align}
  \begin{proof}
From (\ref{f}) we have
\begin{align*}
&\vert f_{J,K}(\phi_{m_1},\ldots, \phi_{m_J})\vert \\
&\leq  \vert\Phi_0(x)\vert\sum_{a=-1}^1\sum_{\sigma\in S_J} \vert w_a\vert \cdot\vert g_{J-1, K+a}(\phi_{m_{\sigma(1)}},\ldots, \phi_{m_{\sigma(J-1)}})\vert
\end{align*}
We then multiply by 
\[
\vert \beta_{m_1}(x)\ldots \beta_{m_J}(x)\vert \leq \vert c_{m_1}\ldots c_{m_J}\vert \tau^J.\]
Summing in $m_1,\ldots, m_J$ completes the proof of (\ref{convergenceJ<p}).

For $J=p$, we multiply instead by

\[
\int_0^\infty\vert \beta_{m_1}(x)\ldots \beta_{m_J}(x)\vert dx \leq \vert c_{m_1}\ldots c_{m_J}\vert \mathfrak m^J,\]
to get $(\ref{convergenceJ=p})$.

  \end{proof}
  
   \begin{proof}[Proof of Lemma \ref{mainlemma}.]
   We sum (\ref{induction}) in $J=1,\ldots p-1,$ to obtain 
   \begin{equation}
   \left\vert \int_a^b\left( \mathcal S_{1,1}(x)-\sum_{K=1}^p \mathcal S_{p,K}(x)-\sum_{j=2}^p \mathcal S_{j,0} (x)\right) dx\right\vert \leq \sum_{j=1}^{p-1} \sum_{l=1}^j \frac{1}{l} E_{j,l}\tau^j.
   \end{equation}
   Note that the RHS converges due to the assumption (\ref{smalldivisor}) together with Lemma \ref{h-glemma}. By using Lemma \ref{convergence} for $J=p$, integrating in $x$ and summing in $K$, 
   \[
   \left \vert \sum_{K=1}^p \int_a^b \mathcal S_{p,K}(x) dx\right\vert \leq \norm \Phi_0\norm \sum_{r=0}^{p-1} E_{p-1,r}\sum_{l=1}^\infty \vert c_l\vert  \mathfrak m^p .\]
We have now that
\begin{align*}
\vert \ln R(b)-\ln R(a)\vert-\tilde B(b) \leq & \sum_{j=1}^{p-1} \sum_{l=1}^j \frac{1}{l} E_{j,l}\tau^j\\
&+\norm\Phi_0\norm\sum_{r=0}^{p-1} E_{p-1,r}\sum_{l=1}^\infty \vert c_l\vert  \mathfrak m^p. 
\end{align*}
where $\tilde B(b)$, a bound on the $\sum \mathcal S_{j,0}$ term, is independent of our choices of $R$ and $\eta$. In other words, we can write

\[ \ln R(b)-\ln R(a)+i(\eta(b)-\eta(a)) =A(b)+B(b)\]
Where $A(b)$ converges as $b\to\infty$, and $B(b)$ is independent of $\eta$ and $R$. Consider now two solutions of (\ref{perturbation}), $u_1, u_2$ with the corresponding $(R_1,\eta_1), (R_2,\eta_2)$. We then also have
\begin{align*}
\ln R_1(b)-\ln R_1(a)+i(\eta_1(b)-\eta_1(a)) =A_1(b)+B_1(b)\\
\ln R_2(b)-\ln R_2(a)+i(\eta_2(b)-\eta_2(a)) =A_2(b)+B_2(b)
\end{align*}
We note that $B_1=B_2$. Therefore subtracting the first equation from the second we obtain
\[\ln R_2(b)+i\eta_2(b)-\ln R_1(b)-i\eta_1(b)=A_2(b)-A_1(b).\]
In particular, both sides of this equation converge when $b\to\infty$, and so we must know that
\[\ln \frac{R_2(b)}{R_1(b)}, \eta_1(b)-\eta_2(b),\]
both converge as $b\to\infty$. 

However, we know that the Wronskian of $u_1, u_2$ does not depend on $b$. We may express the Wronskian as
\[\omega=R_1(b)R_2(b)\sin \vert \eta_1(b)-\eta_2(b)\vert\]

   But we know that $(\eta_1(b)-\eta_2(b))$ converges as $b\to\infty$, and so in response to a choice of $u_2$ it is possible to choose solution $u_1$ so that $\lim_{b\to\infty}\sin \vert \eta_1(b)-\eta_2(b)\vert=\epsilon $ for some $\epsilon>0$. This is because  convergence of $\eta_1(b) -\eta_2(b)$ is at a rate independent of initial conditions. Thus for sufficiently large $b$, we can choose $u_1$ by, say, the initial
condition $\eta_1(b) = \eta_2(b) + \pi /2$, and this would guarantee the
limit is nonzero.  But then we have 
   
   \[\lim_{ b\to\infty} \ln (R_2(b)^2)=-\ln ((\epsilon)/\omega)+ \lim_ {b\to\infty} \mathrm{Re}(A_2(b)-A_1(b)).\]
  
  Thus $R_2(b)$ converges, and hence we have proven our lemma.
     \end{proof}
\end{section}

\begin{section}{Proofs of theorems}

\begin{lemma}\label{Hausdorfflemma}
Assume that (\ref{alphacondition}) holds. Then for a positive integer $j$, the set of $k$ for which the condition (\ref{smalldivisor}) fails has Hausdorff dimension at most $j\mathfrak a$.
\end{lemma}
\begin{proof}
The proof is similar to that of Lemma 4.2 in \cite{Lukic-infinite}, even though our $h$ is defined slightly differently. The most significant difference is that our singularities are at    $2k-\sum \phi=2\pi n$ rather than just at $0$, so each choice of $\sum\phi$ generates infinitely many singularities rather than just one. We adjust the proof by restricting the measure $\nu$ in Lemma 4.1 of \cite{Lukic-infinite} to be a finite uniformly $\beta$-H\"older continuous measure on $[-\pi,\pi]$. 
\end{proof}

\begin{proof}[Proof of Theorem \ref{infinitefreqthm}]
Note that by standard results in Floquet Theory (cf. \cite{Weidmann}) the quasimomentum $k$ is monotone and analytic on bands of the ac spectrum of the unperturbed operator. Thus the theorem follows immediately from Lemmas \ref{mainlemma} and \ref{Hausdorfflemma}, and the fact that monotone analytic maps preserve Hausdorff dimension.
\end{proof}

\begin{proof}[Proof of Theorem \ref{finitefreqthm}]
It is clear that for finite frequencies, the points of our $S$ are the only ones which might not satisfy the small divisor condition (\ref{smalldivisor}).
\end{proof}
\end{section}
\begin{section}{Existence of embedded eigenvalues}

We already know that the set $S$ described in Theorem \ref{finitefreqthm} is optimal, since they are optimal for $V_0=0$, by \cite{Kruger} and \cite{Lukic-continuous}. In this section we wish to demonstrate examples of point spectrum even when the background potential $V_0$ is not identically zero.

The proofs in this section will be similar to the proofs in Section 6 of \cite{Lukic-continuous}, except that $f$ in our proofs are periodic in $x$ instead of constant in $x$. The following lemma will thus prove useful.

\begin{lemma}\label{periodicization}
Let $P(x)$ be a $C^1$ $1$-periodic function on $\R_+$, and let $\mathcal P$ be its mean. Let $q(x)$ be a $C^1$ function of bounded variation on $\R_+$ such that $q'(x)\in L^1(\R_+)$ and $\lim_{x\to\infty}q(x)=0$. Then $\int_0^\infty (P(x)-\mathcal P) q(x) dx$ is finite.
\end{lemma}
\begin{proof}
This follows from integration by parts. Let $A(x)$ be an antiderivative of $P(x)-\mathcal P$ and notice that $A(x)$ is periodic. We then calculate

\begin{align*}
\int_0^\infty (P(x)-\mathcal P) q(x) dx= A(x)q(x)\bigg|_0^\infty-\int_0^\infty A(x)q'(x) dx,&
\end{align*}
and observe that the term on the RHS is finite.

\end{proof}

\begin{lemma}\label{embeddedlemma}
Let $R(x)$, $\eta(x)$ be the Pr\"ufer variables corresponding to some solution of (\ref{perturbation}). Assume that
\begin{equation}
\dfrac{d}{dx} \log R(x)\sim -\frac{B(x)}{x^{(p-1)\gamma}},
\end{equation}
for some periodic $C^1$ function $B(x)$ with positive mean,
and the limit 
\[\eta_\infty=\lim_{x\to\infty}\eta(x)\]
 exists. Then for some $A>0$, 
\[u(x)=Af(x)e^{i[kx+\eta_\infty]}(1+o(1)), x\to\infty.\]

where denoting $\mathcal B$ as the mean of $B(x)$ (remember that $\mathcal B$ is positive),
\[f(x)=\begin{cases}
x^{-\mathcal B} & \gamma=\frac{1}{p-1}\\
\exp\left( -\frac{\mathcal B}{1-(p-1)\gamma} x^{1-(p-1)\gamma} \right) &\gamma\in \left( \frac{1}{p},\frac{1}{p-1}\right)
\end{cases}\]
These asymptotics imply the existence of an $L^2$ solution of (\ref{perturbation}) if $\gamma\in\left( \frac{1}{p},\frac{1}{p-1}\right)$ , and hence an eigenvalue.
\end{lemma}
\begin{proof}
This follows immediately Lemma 6.1 in \cite{Lukic-continuous} and our Lemma \ref{periodicization}.
\end{proof}

\begin{thm}\label{embeddedeigenvalues}

Consider
\begin{equation}\label{examplepotential}
V(x)=\sum_{l=1}^K  L_k \frac{1}{x^\gamma}\cos(\alpha_lx+\xi_l(x))+\beta_0(x), x\geq x_0
\end{equation}
where 
\[\gamma\in \left( \frac{1}{p},\frac{1}{p-1}\right],\]
$L_k>0$, and 
\begin{equation}\label{betaconditions}
\beta_0(x)\in C^1, \dfrac{d}{dx}(\beta_0(x))=O(x^{-p\gamma}),  \beta_0(x)=O(x^{-\gamma}), x\to\infty.
\end{equation}

The functions $\xi_l(x)\in C^1$ have the property that

\[\xi_l'(x)=O(x^{-(p-1)\gamma}), x\to\infty.\]
If $\beta_0(x)$ has bounded variation, this ensures that (\ref{examplepotential}) has generalized bounded variation with phases \[\{0, \pm \alpha_1, \ldots, \pm \alpha_K\}.\] Thus $V$ is $L^p$. Our values for $\phi_1,\phi_2,\ldots$ are then drawn from $\{0, \pm \alpha_1, \ldots, \pm \alpha_K\}$.

Consider a value of $k$ for which $2k\equiv\phi_{j_1}+\ldots, +\phi_{j_{p-1}}\mod 2\pi$, such that $2k$ cannot be written similarly as a sum of fewer phases.
If the $1$-periodic function $f_{p-1,1}(x;\phi_{j_1},\ldots, \phi_{j_{p-1}})e^{2i\varpi(x)}$ does not have mean $0$, then there are choices of $\beta_0$ and $\xi_l$ so that the operator $H$ given by (\ref{perturbation}) has point spectrum at all energies $E$ with the given quasimomentum $k$.
\end{thm}

We will remark that the methods of the previous section make clear that the converse is true, i.e. if $f_{p-1,1}(x;\phi_{j_1},\ldots, \phi_{j_{p-1}})$ has mean zero, then there is no point spectrum at the specified value of the quasimomentum.

\begin{proof}[Proof of Theorem \ref{embeddedeigenvalues}]
This proof will follow closely the proof of Theorem 1.2. of \cite{Lukic-continuous}. 

We start from (\ref{R}) and apply the iterative algorithm in the previous section. Recall that the algorithm could not deal with the term

\[f_{p-1,1}(\phi_{j_1},\ldots, \phi_{j_{p-1}})\beta_{j_1}(x)\ldots\beta_{j_{p-1}}(x)e^{2ikx+2i\varpi(x)+2i\eta(x)},\]
and that instead we had to bound it separately in the form of Lemma \ref{h-glemma}. Thus if we denote the number of distinct permutations of $(j_1,\ldots, j_{p-1})$ by $C_1$, we obtain

\begin{equation}\label{d R}
\dfrac{d}{dx}\log R(x)\sim \mathrm{Im} \left( \frac{\Lambda(x) }{x^{(p-1)\gamma}}e^{i\xi(x)+2i\eta(x)}\right),
\end{equation}
where 
\[\Lambda(x)=C_1 f_{p-1,1}(x;\phi_{j_1},\ldots \phi_{j_p})e^{2i\varpi(x)}L_{j_1}\ldots L_{j_p},\]

and 

\[\xi(x)=\xi_{j_1}(x)+\ldots+\xi_{j_{p-1}(x)}.\]

Conversely, once we have an appropriate $\xi(x)$, we can construct $\xi_j(x)$ by taking $\xi_j(x)=c_j\xi(x)$, where the $c_j$ are real numbers such that $c_{j_1}+\ldots+c_{j_{p-1}}=1$.

We now need to show that $\eta(x)$ has a limit as $x\to\infty$. We apply (\ref{eta}) and see that
\begin{align*}\dfrac{d\eta}{dx}=&\Phi_0(x)V(x)\mathrm{Re}(-1+e^{2ikx+2i\varpi(x)+2i\eta(x)}         ) \\
\sim&\mathrm{Re}\left(\Omega(x)+ \frac{\Lambda(x) }{x^{(p-1)\gamma}}e^{i\xi(x)+2i\eta(x)}\right) ,
\end{align*}
with 
\[\Omega(x)=\sum_{I=1}^{p-1}\sum_{\phi_{j_1}+\ldots+\phi_{j_I}\in 2\pi\mathbb Z} f_{I,0}(x;\phi_{j_1},\ldots,\phi_{j_I})\beta_{j_1}(x)\ldots\beta_{j_I}(x)).\]

Let us replace every function $f$ with its mean in the definition of $\Omega(x)$, to obtain $\hat\Omega(x)$. Observe that by Lemma 6.2 of \cite{Lukic-continuous}, there is a choice of $\beta_0(x)$ such that $\int_0^\infty \hat\Omega(x)dx$ is finite . By applying Lemma \ref{periodicization}, we can see that $\int_0^\infty \Omega(x)dx$ is finite as well.

We then have 
\begin{equation}\label{d eta}
\dfrac{d\eta}{dx}\sim \frac{\Lambda(x)}{x^{(p-1)\gamma}}.\end{equation}

Firstly, since we assumed $f_{p-1,1}(x)$ and hence $\Lambda(x)$ has nonzero mean, it must be true that $\mathrm{Im}( \Lambda(x) e^{it})$ has positive mean for some real $t$. We denote $\psi(x)=\xi(x)+2\eta(x)$. 

\begin{lemma}\label{pickpsi}
Let $R(x), \eta(x)$ be Pr\"ufer variables corresponding to some solution of (\ref{perturbation}). Assume that (\ref{d R}) and (\ref{d eta}) hold. Then we may pick $\xi(x)$ with $\xi'(x)\in O(x^{-(p-1)\gamma})$ such that $\lim_{x\to\infty}\psi(x)=t$.  
\end{lemma}
\begin{proof}The proof is identical to that of Lemma 6.3 of \cite{Lukic-continuous}, except that we replace the constant $\Lambda$ with the periodic function $\Lambda(x)$, and their Lemma 4.1 with our Lemma \ref{fubini}.
\end{proof}

With that choice of $\xi(x)$, 
\begin{align*}
\dfrac{d}{dx}\log R(x)\sim &\mathrm{Im}\left(\frac{\Lambda(x)}{x^{(p-1)\gamma}}e^{i\psi_\infty}+\frac{\Lambda(x)}{x^{(p-1)\gamma}}(e^{i\psi(x)}-e^{i\psi_\infty})\right)\\
\sim& \mathrm{Im}\left(\frac{\Lambda(x)e^{it}}{x^{(p-1)\gamma}}+O(x^{-p\gamma})\right)\\
\sim& \mathrm{Im}\left(\frac{\Lambda(x)e^{it}}{x^{(p-1)\gamma}}\right),
\end{align*}
and then we simply apply Lemma \ref{embeddedlemma} to complete our proof.
\end{proof}

We now need to determine how often  the condition that \[f_{p-1,1}(x; \phi_{j_1},\ldots\phi_{j_{p-1}})e^{2i\varpi(x)}\] has nonzero mean is satisfied. We will show that this condition is satisfied for a nontrivial class of periodic functions $\varphi(x)e^{-ikx}$. We will start with a suggestive example. For notational convenience, let us adjust the order of the phases so we can rewrite $\phi_{j_1}, \phi_{j_2}, \phi_{j_{p-1}}$ as $\phi_1,\phi_2,\ldots,\phi_{p-1}$.

\begin{prop}\label{codimension1example}
Assume that the Floquet solution $\varphi$ is given as $Ce^{ikx}$ for some positive $C$ (i.e., $\Phi_0(x)=\frac{ C }{\omega}$ and $\varpi(x)=0$).
Then for a choice of phases $\{\alpha_i\}$ away from an algebraic set of codimension $1$, for every $1\leq l\leq j\leq p-1$ \[f_{j,l}(x; \phi_{1},\ldots\phi_{j})e^{2il\varpi(x)},\]
is a nonzero constant in $x$ (the constant depends on $j,l$).
\end{prop}
\begin{proof}
Let us assume that $2Kk-\sum_{t=1}^{l} \phi_{\sigma(t)}\equiv 0\mod 2\pi$ does not hold for any $l\leq j$, any choice of phases, any permutation $\sigma$ of $p-1$ elements, and any $K<p$. We can make this assumption since it is a codimension $1$ condition on the $\{\alpha_i\}$.

We may calculate that for a constant $C$, $e^{2il\varpi(x)}\lambda_{\alpha, l} C=-\frac{C}{\alpha}$. Thus applying (\ref{f}), we discover that 
\[f_{j,l}(x; \phi_{1},\ldots\phi_{j})e^{2il\varpi(x)}\]
is a rational function in the variables $k,\phi_1,\ldots, \phi_{j}$, with denominator terms of the form $2Kk-\sum \phi$. Note that for large enough $k$ all the terms are strictly positive (the quasimomentum $k$ only takes values in a $\pi$-interval, but in the context of this proof we are viewing it as a variable in $\R$) . But then it is an easy induction argument that \[(-1)^{l+1}f_{j,l}(x; \phi_{1},\ldots\phi_{j})e^{2il\varpi(x)}\]
is strictly positive, using (\ref{g}) and (\ref{f}). This demonstrates that it is a nontrivial rational function, and therefore is only zero on a set of $\phi_j$s of codimension $1$. 
\end{proof}
\begin{lemma}\label{Fourierlemma}
Assume that the $1$-periodic function $\varphi(x)e^{-ikx}$ has finite Fourier expansion
\[ \sum_{n=-N}^N \hat\varphi(n)e^{2\pi i n x}.\]
Assume that the phases $\{\alpha_i\}$ are chosen away from the codimension $1$ algebraic set described in Proposition \ref{codimension1example}. Then if the Fourier coefficients $\hat \varphi(n)$ are chosen away from another algebraic set of codimension $1$, the corresponding
\[f_{j,l}(x; \phi_{1},\ldots\phi_{j})e^{2il\varpi(x)},\]
have nonzero mean for all $1\le l \le j \le p-1$.
\end{lemma}
\begin{proof}
As a first step, we have to understand how $e^{2il\varpi(x)}\lambda_{\alpha, l}$ acts on finite Fourier sums. So let us consider $\Phi(x)=\sum_{n=-N}^N \hat \Phi(n) e^{2\pi i n x}.$ For the reader's convenience, we will take this somewhat tedious calculation step by step, using the proof and notation of Proposition \ref{lambda} as a guide.

First, we need to determine the value of $Q_\alpha(x)$, that is, the antiderivative of $\Phi(x)e^{i\alpha x}$ with $Q_\alpha(0)=0$. 
We find that 
\begin{align*}
Q_\alpha(x)=&\int \sum_{n=-N}^N \hat \Phi(n) e^{(2\pi  n+\alpha) ix} dx\\
=& C+\sum_{n=-N}^N \hat \Phi(n) \frac{e^{(2\pi  n+\alpha) ix}}{(2\pi n+\alpha)i}.
\end{align*}
Using $Q(0)=0$, it is easy to calculate the value of $C$. We then obtain, finally
\[
Q_\alpha(x)=\sum_{n=-N}^N \hat \Phi(n) \frac{e^{(2\pi  n+\alpha) ix}-1}{(2\pi n+\alpha)i}.
\]
Then, by (\ref{tildePhi}), we have
\begin{align*}
\tilde\Phi_\alpha(x)=&- \left(\sum_{n=-N}^N \hat \Phi(n) \frac{e^{(2\pi  n+\alpha) ix}-1}{(2\pi n+\alpha)}\right)e^{-i\alpha x}(1-e^{i\alpha})\\
&+ \left(\sum_{n=-N}^N \hat \Phi(n) \frac{e^{(2\pi  n+\alpha)i }-1}{(2\pi n+\alpha)}\right)e^{-i\alpha x}\\
=&- \left(\sum_{n=-N}^N \hat \Phi(n) \frac{e^{2\pi  n ix}-e^{-i\alpha x}}{(2\pi n+\alpha)}\right)(1-e^{i\alpha})\\
&+ \left(\sum_{n=-N}^N \hat \Phi(n) \frac{e^{i\alpha (1-x)}-e^{-i\alpha x}}{(2\pi n+\alpha)}\right)\\
=& \sum_{n=-N}^N  -\frac{\hat\Phi(n)(1-e^{i\alpha})}{(2\pi n+\alpha)} e^{2\pi  n ix}.
\end{align*}
and thus by Definition \ref{lambdadefn}, we know that \[\Phi(x)\to e^{2il\varpi(x)}\frac{\lambda_{\alpha,l}\Phi(x)}{1-e^{i\alpha}}\]
modifies the Fourier coefficients so that $\hat\Phi(n)\to -\frac{\hat\Phi(n)}{2\pi n+\alpha}$.

Now using (\ref{g}),(\ref{f}), and the fact that \[\Phi_0(x)=\frac{\sum_{n=-N}^N \hat\varphi(n)e^{2\pi i n x}\sum_{n=-N}^N \overline{\hat\varphi(-n)}e^{2\pi i n x}}{\omega},\]
we deduce that 
\[f_{j,l}(x; \phi_{1},\ldots\phi_{j})e^{2il\varpi(x)},\]
is a finite Fourier sum, whose coefficients are all polynomials in $\{\hat\varphi(n)\}\cup\{\overline{\hat\varphi(n)}\}$. In particular, the zeroth Fourier coefficient, which gives us the mean of $f_{j,l}$, is a polynomial in $\{\hat\varphi(n)\}\cup\{\overline{\hat\varphi(n)}\}$. But we know it is not identically zero, since by Proposition \ref{codimension1example} the mean is nonzero when $\hat\varphi(0)=\overline{\hat \varphi(0)}=C$ a positive constant and all the other $\hat\varphi(n),\overline{\hat \varphi(n)}$ are zero. Thus the mean can be zero only on a codimension $1$ set of $\{\hat\varphi(n)\}\cup\{\overline{\hat\varphi(n)}\}$.
\end{proof}
\begin{prop}
Assume that the phases $\{\alpha_i\}$ are chosen away from the codimension $1$ algebraic set described in Proposition \ref{codimension1example}. Then for a dense open set of $V_0(x)$ in (\ref{unperturbed}) in the $L^1(0,1)$-topology , the corresponding
\[f_{j,l}(x; \phi_{1},\ldots\phi_{j})e^{2il\varpi(x)},\]
have nonzero mean for all $1\le l \le j \le p-1$.
\end{prop}
\begin{proof}
Let us for notational convenience set $\psi(x)=\varphi(x)e^{-ikx}$. First note that trigonometric polynomials are dense in the space of $1$-periodic functions under the $W^{2,1}((0,1))$ topology. We further claim that the set of trigonometric polynomials in Lemma \ref{Fourierlemma} (i.e., missing an algebraic codimension $1$ set) is still dense in the space of $1$-periodic functions. This is obvious, since for any $n$ we can apply an arbitrarily small trigonometric polynomial (in the $W^{2,1}((0,1))$ sense) perturbation of degree $n$ to a trigonometric polynomial in that algebraic codimension $1$ set so that the perturbed polynomial is not in the codimension $1$ set. 

Also, the condition that the
\[f_{j,l}(x; \phi_{1},\ldots\phi_{j})e^{2il\varpi(x)},\]
have nonzero mean for all $1\le l \le j \le p-1$ is clearly an open condition in $W^{2,1}((0,1))$-norm of $\psi(x)$, since the expressions are sums of antiderivatives of $\psi$. Thus this condition is an open and dense condition in the space of $1$-periodic functions $\psi(x)$ under the $W^{2,1}((0,1))$ topology.

Furthermore, if we write (\ref{unperturbed}) in terms of $\psi$ and $k$, we get

\[V_0(x)-E=\frac{\psi''(x)}{\psi(x)}+2ik\frac{\psi'(x)}{\psi(x)}-k^2.\]
Recall that since we assumed $\varphi$, $\overline\varphi$ are linearly independent, it must be true that $\psi(x)$ is nonzero for all $x$.
So, noting that the quasimomentum $k$ depends continuously on $V_0$, it is clear that an open and dense set in $\psi$ corresponds to an open and dense set in $V_0$ (using the $L^1(0,1)$ topology). 

\end{proof}
Theorem \ref{embeddedthm} is an immediate corollary of this proposition.
\end{section}
\bibliographystyle{alpha}   
\bibliography{mybib}
\end{document}